\documentclass[]{amsart}
\usepackage{amscd,amsthm,amssymb,amsfonts,amsmath,euscript}

\theoremstyle{plain}
\newtheorem{thm}{Theorem}[section]
\newtheorem{lemma}[thm]{Lemma}

\newtheorem{cor}[thm]{Corollary}

\theoremstyle{definition}

\theoremstyle{remark}
\newtheorem{remark}[thm]{Remark}

\newcommand{\nc}{\newcommand}

\def\makeop#1{\expandafter\def\csname#1\endcsname
  {\mathop{\rm #1}\nolimits}\ignorespaces}
\makeop{Hom}   \makeop{End}   \makeop{Aut}   \makeop{Isom}  \makeop{Pic} 
\makeop{Gal}   \makeop{ord}   \makeop{Char}  \makeop{Div}   \makeop{Lie} 
\makeop{PGL}   \makeop{Corr}  \makeop{PSL}   \makeop{sgn}   \makeop{Spf}
\makeop{Spec}  \makeop{Tr}    \makeop{Nr}    \makeop{Fr}    \makeop{disc}
\makeop{Proj}  \makeop{supp}  \makeop{ker}   \makeop{im}    \makeop{dom}
\makeop{coker} \makeop{Stab}  \makeop{SO}    \makeop{SL}    \makeop{SL}
\makeop{Cl}    \makeop{cond}  \makeop{Br}    \makeop{inv}   \makeop{rank}
\makeop{id}    \makeop{Fil}   \makeop{Frac}  \makeop{GL}    \makeop{SU}
\makeop{Nrd}   \makeop{Sp}    \makeop{Tr}    \makeop{Trd}   \makeop{diag}
\makeop{Res}   \makeop{ind}   \makeop{depth} \makeop{Tr}    \makeop{st}
\makeop{Ad}    \makeop{Int}   \makeop{tr}    \makeop{Sym}   \makeop{can}
\makeop{length}\makeop{SO}    \makeop{torsion} \makeop{GSp} \makeop{Ker}
\makeop{Adm}   \makeop{Mat}
\def\makebb#1{\expandafter\def
  \csname bb#1\endcsname{{\mathbb{#1}}}\ignorespaces}
\def\makebf#1{\expandafter\def\csname bf#1\endcsname{{\bf
      #1}}\ignorespaces} 
\def\makegr#1{\expandafter\def
  \csname gr#1\endcsname{{\mathfrak{#1}}}\ignorespaces}
\def\makescr#1{\expandafter\def
  \csname scr#1\endcsname{{\EuScript{#1}}}\ignorespaces}
\def\makecal#1{\expandafter\def\csname cal#1\endcsname{{\mathcal
      #1}}\ignorespaces} 

\def\doLetters#1{#1A #1B #1C #1D #1E #1F #1G #1H #1I #1J #1K #1L #1M
                 #1N #1O #1P #1Q #1R #1S #1T #1U #1V #1W #1X #1Y #1Z}
\def\doletters#1{#1a #1b #1c #1d #1e #1f #1g #1h #1i #1j #1k #1l #1m
                 #1n #1o #1p #1q #1r #1s #1t #1u #1v #1w #1x #1y #1z}
\doLetters\makebb   \doLetters\makecal  \doLetters\makebf
\doLetters\makescr 
\doletters\makebf   \doLetters\makegr   \doletters\makegr

     \def\qed{\qedmark\medbreak}%
\def\qedmark{{\enspace\vrule height 6pt width 5pt depth 1.5pt}}%

\normalsize

\makeop{Bl}

\def\Spec{{\rm Spec}\,}

\def\Fpbar{\overline{\bbF}_p}
\def\Fp{{\bbF}_p}
\def\Fq{{\bbF}_q}

\def\Qp{{\bbQ}_p}

\def\Zp{{\bbZ}_p}

\newcommand{\Z}{\mathbb Z}
\newcommand{\Q}{\mathbb Q}

\newcommand{\F}{\mathbb F}

\newcommand{\npr}{\noindent }

\newcommand{\<}{\langle}   
\renewcommand{\>}{\rangle} 

\nc{\embed}{\hookrightarrow}

\newcommand{\dieu}{Dieudonn\'{e} }

\nc{\ol}{\overline}
\nc{\wt}{\widetilde}
\nc{\opp}{\mathrm{opp}}

\makeop{Ram}
\makeop{Rep}

\def\sfF{\mathsf{F}}
\def\sfV{\mathsf{V}}

\begin{document}
\renewcommand{\thefootnote}{\fnsymbol{footnote}}
\setcounter{footnote}{-1}
\numberwithin{equation}{section}


\title[Fields of definition of components]
  {On fields of definition of components of the Siegel supersingular
  locus}
\author{Chia-Fu Yu}
\address{
Institute of Mathematics, Academia Sinica \\
6th Floor, Astronomy Mathematics Building \\
No. 1, Roosevelt Rd. Sec. 4 \\ 
Taipei, Taiwan, 10617} 
\email{chiafu@math.sinica.edu.tw}

\address{
 National Center for Theoretical Sciences
 No.~1  Roosevelt Rd. Sec.~4,
 National Taiwan University
 Taipei, Taiwan, 10617}

\date{\today} \subjclass[2010]{11R52, 11G10} \keywords{superspecial
  abelian surface, supersingular locus}

\begin{abstract}
Recently Ibukiyama proves an explicit formula for the number of 
certain non-principal polarizations on a superspecial abelian surface, 
extending his earlier work with Katsura for principal polarizations
[Compos. Math. 1994]. 
As a consequence of Ibukiyama's formula, there exists a
geometrically irreducible component of the Siegel supersingular locus
which is defined over the prime finite field. In this note we give a
direct proof of this result.    
\end{abstract} 

\maketitle


\section{Introduction}
\label{sec:01}

Let $p$ be a rational prime number, and let $n\ge 1$ be a positive
integer.
Let $\calA_n$ denote the coarse moduli space over $\Fp$ of $n$-dimensional
principally polarized abelian varieties.
The supersingular locus of $\calA_n\otimes \Fpbar$ is denoted by $\ol
\calS_n$, which is the closed reduced $\Fpbar$-subscheme consisting of all
supersingular points in $\calA_n(\Fpbar)$. An abelian
variety over $\Fpbar$ is {\it supersingular} (resp.~{\it
  superspecial}) 
if it is 
isogenous (resp. isomorphic) to a product of supersingular elliptic curves.
It is known that $\ol \calS_n$ is defined over $\Fp$,
i.e. the action of the Galois group $\Gamma:=\Gal(\Fpbar/\Fp)$ 
leaves the set $\ol \calS_n$ stable. 
The unique model of $\ol \calS_n$ over $\Fp$ in $\calA_n$ 
is denoted by $\calS_n$, which is called the supersingular locus of
$\calA_n$. The set of irreducible components of $\ol \calS_n$ is denoted by
$\Pi_0(\ol \calS_n)$, on which $\Gamma$ operates.
An irreducible component $V\in \Pi_0(\ol \calS_n)$ is defined over $\Fp$
if and only if $V$ is stable under the action of $\Gamma$. 
In this note we give a proof of the following result.

\begin{thm}[Li-Oort, Ibukiyama]\label{1.1}
  There exists an irreducible component of $\ol \calS_n$ that is defined
  over $\Fp$. 
\end{thm}

In Section 2, 
we shall give some background knowledge on $\Pi_0(\ol \calS_n)$ due to
Li and Oort \cite{li-oort}.
Based on loc.~cit., Theorem~\ref{1.1} is reduced to 
non-emptiness of the set of certain polarized superspecial 
abelian varieties that admit a model defined over $\Fp$ 
(the set $\Lambda_n(\Fp)$ or $\Sigma_n(\Fp)$ in 
Subsection~\ref{sec:2.1}); see Theorem~\ref{2.3}. 
It follows that Theorem~\ref{1.1} is trivial when $n$ is odd. 
When $n$ is even, Theorem~\ref{1.1} then follows from the following
result.

\begin{thm}\label{1.2}
  There exists a polarized superspecial abelian surface $(A,\lambda)$
  over $\Fp$ such that $\ker \lambda = A [F]$, 
  where $F:A\to A^{(p)}=A$ is the relative
  Frobenius morphism on $A$. 
\end{thm}

In \cite{ibukiyama:quinary,ibukiyama:type} Ibukiyama gives an explicit
formula for the cardinality of $\Sigma_2(\Fp)$,
and as a byproduct he obtains Theorem~\ref{1.2}. 
The result of Ibukiyama confirms the existence
of polarized abelian surfaces as in Theorem~\ref{1.2}. We construct such an
polarized abelian surface directly in Section~\ref{sec:3}.  
This proves Theorem~\ref{1.1} by a different method.

\section{Preliminaries and background}
\label{sec:02}


\subsection{Finite sets $\Lambda_n$ and $\Sigma_n$}
\label{sec:2.1}

Let $\Gamma:=\Gal(\Fpbar/\Fp)$ denote the Galois group over $\Fp$, and
put $\Gamma_2:=\Gal(\Fpbar/\F_{p^2})$. 
The Frobenius automorphism $a\mapsto a^p$ in $\Gamma$ is denoted by
$\sigma_p$. 
For any positive integer $n\ge 1$, let $\Lambda_n$ denote the set of
isomorphism classes of superspecial principally polarized abelian
varieties of dimension $n$ over $\Fpbar$. This is a finite set on
which $\Gamma$ acts. It is known that the subgroup $\Gamma_2$ acts
trivially on $\Lambda_n$ so the action factors through
the quotient $\Gal(\F_{p^2}/\Fp)$. Put   
\[ \Lambda_n(\Fp):=\Lambda_n^\Gamma\subset \Lambda_n, \]
the subset of elements fixed by $\sigma_p$. Let $E_0$ be a
supersingular elliptic curve over $\Fp$ such that the endomorphism
endomorphism 
$\pi_{E_0}$ of $E_0$ satisfies $\pi_{E_0}^2=-p$. Put $A_0=E_0^n$ and take
$\lambda_0$ to be the product of copies of 
the canonical principal polarization of $E_0$. 
Clearly, $(A_0,\lambda_0)$
is a superspecial principally polarized abelian variety of dimension
$n$ over $\Fp$. In particular, the set $\Lambda_n(\Fp)$ is nonempty.

For any even positive integer $n=2c$, denote by $\Sigma_{n}$ the set of
isomorphism classes of superspecial polarized abelian varieties
$(A,\lambda)$ of dimension $n$ over $\Fpbar$ such that 
$\ker \lambda=A[F]$, where $A^{(p)}$ is the base change of $\Spec
\Fpbar\to \Spec \Fpbar$ induced by $\sigma_p$, and 
$F:A\to A^{(p)}$ is the relative Frobenius morphism.
Clearly, for any $(A,\lambda)\in \Sigma_{n}$, one has $\deg
\lambda=p^n$. The Galois group $\Gamma$ acts on $\Sigma_{n}$ which
factors through $\Gal(\F_{p^2}/\Fp)$. Indeed, any superspecial abelian
variety over $\Fpbar$ admits a model $A'_0$ over $\F_{p^2}$ 
such that any endomorphism of $A$ is defined over $\F_{p^2}$ (for the
$\F_{p^2}$-structure of $A_0'$), in particular, any
polarization of $A$ is defined over $\F_{p^2}$. Thus, the group
$\Gamma_2$ acts trivially on $\Sigma_n$. Similarly, we define  
\[ \Sigma_n(\Fp):=\Sigma_n^\Gamma\subset \Sigma_n, \]
the subset of elements fixed by $\sigma_p$. Unlike $\Lambda_n$, it is
not clear whether $\Sigma_n(\Fp)$ is nonempty a priori. 

Recall that the field of moduli of a polarized abelian variety
$(A,\lambda)$ over $\Fpbar$ is the field of definition of the
isomorphism class $[(A,\lambda)]$, that is, the smallest 
finite field $\F_{p^a}$ 
such that $(A,\lambda)\simeq ({}^\sigma A, {}^\sigma \lambda)$ over
$\Fpbar$ for any $\sigma\in \Gal(\Fpbar/\F_{p^a})$. 

\begin{lemma}\label{2.1}
  Let $(A,\lambda)$  be any polarized abelian variety over
  $\Fpbar$, and suppose that the field of moduli of $(A,\lambda)$ is 
  $\F_{p^a}$. 
  Then $(A,\lambda)$  has a model defined over
  $\F_{p^a}$. 
\end{lemma}
\begin{proof}
  The lemma is proved for $A$ being superspecial \cite[Lemma
  4.1]{ibukiyama:type}. The same proof works for an arbitrary abelian
  variety; also see \cite[Prop.~6.3]{yu:superspecial} for a similar proof. \qed 
\end{proof}

\begin{lemma}\label{2.2}
  If an element $(A,\lambda)$ in $\Sigma_n$ has a model $(A',\lambda')$
  defined over $\Fp$, then $\ker \lambda'=A'[F]$. 
\end{lemma}
\begin{proof}
  We use the following basic fact: if $H_1 < H_2$ are two finite
  group schemes over a field and if $\rank H_1=\rank H_2$, then
  $H_1=H_2$. It follows that if $\varphi: A\to B$ is an isogeny of
  abelian varieties over $k$, $K/k$ is a field extension and
  $\varphi_K:A_K\to 
  B_K$ is the base change morphism, then
  $\ker \varphi_K=\ker \varphi \otimes_k K$. Applying this for
  $F:A\to A^{(p)}$ and $K/k=\Fpbar/\Fp$, we get 
  \[ \ker \lambda'\otimes \Fpbar=\ker \lambda=A[F]=A'[F]\otimes
  \Fpbar. \]
  By the faithfully flat descent, we have $\ker \lambda'=A'[F]$. \qed
\end{proof}

Let $\wt \Sigma_n(\Fp)$ be the set of isomorphism classes of
superspecial polarized abelian varieties $(A,\lambda)$ 
of dimension $n$ ($n$ being even)
over $\Fp$ such that $\ker(\lambda)=A[F]$. It follows from
Lemmas~\ref{2.1} and~{2.2} that the natural map $\wt \Sigma_n(\Fp)\to
\Sigma_n(\Fp)$ is surjective. Particularly, 
the set $\Sigma_n(\Fp)$ is nonempty if and only if so is 
$\wt \Sigma_n(\Fp)$. 

\subsection{Fields of definition for components}
\label{sec:2.2}

As in Section 1, we 
let $\calA_n$ denote the coarse moduli space over $\Fp$ 
of $n$-dimensional
principally polarized abelian varieties, and let $\calS_n\subset \calA_g$ 
be the supersingular locus of 
of $\calA_g$. 
Li-Oort proved \cite[4.9, p.~26]{li-oort} that there is
a one-to-one correspondence between the set $\Pi_0(\ol \calS_n)$ 
of irreducible components
of $\ol \calS_n:=\calS_n\otimes \Fpbar$ and either the set $\Lambda_n$ or the set
$\Sigma_n$ according as $n$ is odd or even. 
The Galois group $\Gamma$ operates on $\Pi_0(\ol \calS_n)$ 
as well as on the sets $\Lambda_n$ and $\Sigma_n$. 
We claim that the correspondence in loc.~cit.~respects the action of
$\Gamma$. 
Let $V\in \Pi_0(\ol \calS_n)$ be an irreducible component. 
It is proved in \cite[4.9 (iii), p.~26]{li-oort} that the subset $U$ of 
supergeneral abelian varieties
in the sense of Li and Oort (i.e. with $a=1$) 
is an open and dense subset in
$V$. For any point $(B,\lambda_B)$ in $U$, 
there is a unique ``PFTQ'' (polarized
flag type quotient) up to equivalence \cite[Chapter 3]{li-oort}
\[ (A_{n-1},\eta_{n-1})\to  (A_{n-2},\eta_{n-2})\to \dots \to
(A_{0},\eta_{0}) \]
such that $(A_{0},\eta_{0}) \simeq (B,\lambda_B)$ and
$(A_{n-1},\eta_{n-1})\simeq (A,p^m \lambda)$ for a superspecial polarized
abelian variety $(A,\lambda)$ that lies in $\Lambda_n$ if $n$ is odd and in
$\Sigma_n$ if $n$ is even. The pair $(A,\lambda)$ depends only on
$U$ up to isomorphism, and corresponds to the component $V$. 
It is then clear from this construction that for any $\sigma\in \Gamma$   
 the conjugate ${}^\sigma U$ of $U$ corresponds to the isomorphism class
 $[({}^\sigma A,{}^\sigma\lambda)]$. This proves the claim. 
Note that the isomorphism class
 $[({}^\sigma A,{}^\sigma\lambda)]$ is defined over $\Fp$ if and only
 if $(A.\lambda)$ admits a model defined over $\Fp$ (Lemma~\ref{2.1}).   
Thus, we have the following consequence of results of Li and Oort. 

\begin{thm}[Li-Oort, cf. {\cite[Theorem 4.4]{ibukiyama:type}}]\label{2.3} \

{\rm (1)} Every irreducible component $V\subset \ol \calS_n$
    is defined over $\F_{p^2}$. 

{\rm (2)} Let $\Pi_0(\ol \calS_n)(\Fp)$ be the set of irreducible
    components of $\ol \calS_n$ that are defined over $\Fp$. Then 
    \begin{equation}
      \label{eq:li-oort}
      \Pi_0(\ol \calS_n)(\Fp)\simeq 
      \begin{cases}
        \Lambda_n(\Fp) & \text{if $n$ is odd;} \\
        \Sigma_n(\Fp)  & \text{if $n$ is even.}
      \end{cases}
    \end{equation}
\end{thm}

\subsection{Theorem~\ref{1.2} implies Theorem~\ref{1.1}}
\label{sec:2.3}

It follows from Theorem~\ref{2.3} that Theorem~\ref{1.1}
is equivalent to that $\Lambda_n(\Fp)$ is nonempty if $n$ is odd and
$\Sigma_n(\Fp)$ is nonempty if $n$ is even. Non-emptiness of
$\Lambda_n(\Fp)$ is obvious. Thus, it reduces to prove 
non-emptiness of $\Sigma_n(\Fp)$. But this following from 
non-emptiness of $\Sigma_2(\Fp)$, which is equivalent to
Theorem~\ref{1.2} by Lemma~\ref{2.1}.

\subsection{Explicit formula for $|\Sigma_2(\Fp)|$}
\label{sec:2.4}
Let $\chi$ be the quadratic character associated to the real quadratic
field $\Q(\sqrt{p})/\Q$, and let $B_{2,\chi}$ be the 
second generalized Bernoulli number. For any algebraic number $\alpha$,
we write $h(\alpha)$ for the class number of the number field
$\Q(\alpha)$. 

\begin{thm}[Ibukiyama~{\cite[Theorem 5.2]{ibukiyama:quinary}}]\label{2.4}
  One has \[ \text{$|\Sigma_2(\Fp)|=1,1,1$, for $p=2,3,5$,
  respectively.} \] 
  For $p\ge 7$, if
  $p\equiv 1 \mod 4$, then 
  \begin{equation}
    \label{eq:2.2}
    \begin{split}
      |\Sigma_2(\Fp)|=& \frac{1}{2^5\cdot 3} \left(9-2 \left
       (\frac{2}{p} \right ) \right ) B_{2,\chi}+\frac{1}{2^4}
       h(\sqrt{-p}) \\
      & +\frac{1}{2^3}h(\sqrt{-2p})+\frac{1}{2^2\cdot 3}\left(3+ \left
       (\frac{2}{p} \right ) \right ) h(\sqrt{-3p});
    \end{split}
  \end{equation}
 if $p\equiv 3 \mod 4$, then 
  \begin{equation}
    \label{eq:2.3}
    \begin{split}
      |\Sigma_2(\Fp)|=& \frac{1}{2^5\cdot 3}  B_{2,\chi}+\frac{1}{2^4}
       \left(1- \left
       (\frac{2}{p} \right )\right ) h(\sqrt{-p}) \\
      & +\frac{1}{2^3} h(\sqrt{-2p})+\frac{1}{2^2\cdot 3}  h(\sqrt{-3p}).
    \end{split}
  \end{equation}
\end{thm}

Non-emptiness of $\Sigma_2(\Fp)$ (or Theorem~\ref{1.2}) 
follows immediately from
Theorem~\ref{2.4}, and hence Theorem~\ref{1.1} follows.    

\section{Construction of certain superspecial polarized abelian
  varieties}
\label{sec:3}


In this section we include some results concerning non-emptiness of
the set $\wt \Sigma_n(\Fp)$. 
The Frobenius
endomorphism of an abelian variety $A$ over $\Fq$ will be denoted by
$\pi_A$. Let $\wt \Sigma'_n(\Fp)$ be the set of isomorphism classes of
superspecial abelian varieties $(A,\lambda)$ of dimension $n$ over
$\Fp$ such that $\pi_A^2=-p$ and $\ker \lambda=A[F]$. It is clear that
non-emptiness of $\wt \Sigma'_n(\Fp)$ implies that of $\wt \Sigma_n(\Fp)$.

\begin{lemma}\label{3.1}
  Assume that the integer $n=2c$ is divisible by $4$. 
  Then there is a superspecial polarized 
  abelian variety $(A',\lambda')$ in $\wt \Sigma'_n(\Fp)$. 
\end{lemma}
\begin{proof}
  Choose a superspecial abelian
  variety $A_1$ over $\F_{p^2}$ of dimension $c$ such that
  $\pi_{A_1}=-p$. Since every polarization of $A'\otimes \Fpbar$ 
  is defined over
  $\F_{p^2}$ and the set $\Sigma_c$ is nonempty, there exists a
  polarization $\lambda_1$ of $A_1$ such that $\ker \lambda_1=A_1[F]$.  
  Take 
  \begin{equation}
    \label{eq:3.1}
   (A',\lambda'):=\Res_{\F_{p^2}/\Fp} (A_1,\lambda_1). 
  \end{equation}
Then $(A',\lambda')$ is a superspecial polarized abelian variety over
$\Fp$ of dimension $n$ such that 
\begin{equation}
  \label{eq:2.4}
  (A',\lambda')\otimes \F_{p^2}=(A_1,\lambda_1)\times
(A_1^{(p)},\lambda_1^{(p)}).
\end{equation}
This gives 
\begin{equation}
  \label{eq:2.5}
  \ker \lambda'=\ker \lambda_1\times \ker
\lambda_1^{(p)}=A_1[F]\times A_1^{(p)}[F]=A'[F].
\end{equation}
From (\ref{eq:2.4}) we get $\pi_{A'}^2=-p$. \qed
\end{proof}

\begin{lemma}\label{3.2}
  Assume that $\left (\frac{-1}{p}\right)=1$. Then for any even positive
  integer $n$, there exists a
  superspecial polarized abelian variety $(A,\lambda)$ in $\wt
  \Sigma'_n(\Fp)$ 
\end{lemma}
\begin{proof}
  It suffices to prove this for $n=2$ as we can take the product of
  copies of such a polarized abelian surface. Let
  $(A_0,\lambda_0)=(E_0^2,\lambda_0)$ be the superspecial principally 
  polarized
  abelian surface over $\Fp$ as in Subsection~\ref{sec:2.1}; one has
  $\pi_{A_0}^2=-p$. Consider isogenies $\alpha:(A,\lambda)\to
  (A_0,\lambda_0)$ of degree $p$ with $\alpha^*
  \lambda_0=\lambda$. The family $\{\alpha\}$ forms a projective space
  $\bfP^1$ that has an $\Fp$-structure induced from the $\Fp$-structure of
  $A_0$. If an isogeny $\alpha:(A,\lambda)\to (A_0,\lambda_0)$
  corresponds to a point $[a:b]\in \bfP^1(\Fpbar)$, then
  $(A,\lambda)\in \Sigma_2$ if and only if $a^{p+1}+b^{p+1}=0$
  (\cite[(3.4), p.~119]{katsura-oort:surface} 
  and \cite[Lemma 4.3]{yu:ss_siegel}). Since
  $(-1/p)=1$, there exists a point $[a:b]\in \bfP^1(\Fp)$ such that
  $a^{p+1}+b^{p+1}=a^2+b^2=0$. Then the corresponding isogeny
  $\alpha:(A,\lambda)\to (A_0,\lambda_0)$ is defined over $\Fp$ and
  satisfies both $\pi_A^2=-p$ and $\ker \lambda=A[F]$. \qed   
\end{proof}

\begin{cor}\label{3.3}
  Assume that $4|n$ or $\left(\frac{-1}{p}\right)=1$. Then the
  set $\wt \Sigma'_n(\Fp)$ is nonempty. Consequently, so is the set
  $\wt \Sigma_n(\Fp)$. 
\end{cor}

\begin{remark}\label{3.4}
We expect that the sufficient conditions in Corollary~\ref{3.3} are
also necessary for non-emptiness of $\wt \Sigma'_n(\Fp)$.  
\end{remark}

For the rest of this section we construct a polarized abelian surface
as in Theorem~\ref{1.2}. \\

\npr {\bf Step 1.} There exists a principally polarized abelian surface
$(A_0,\lambda_0)$ over $\Fp$ with  $\pi_{A_0}^2=p$. 

Choose an supersingular elliptic curve $E$ over $\F_{p^2}$ such that
$\pi_{E}=p$. Take
\[ A_0=\Res_{\F_{p^2}/\Fp} E. \]
Then the Frobenius endomorphism $\pi_{A_0}$ of $A_0$ satisfies 
$\pi_{A_0}^2=p$, and one has $A_0\otimes \F_{p^2}=E\times
{}^{\sigma_p} E$.  
As $E$ admits a principal polarization $\mu$, 
the product $\lambda_0:=\mu\times {}^{\sigma_p} \mu$ on $A_0\otimes
\F_{p^2}$ is a principal polarization that is defined over $\Fp$. \\
 
\npr {\bf Step 2.} 
The (covariant) \dieu module $M_0$ of $A_0$ is a free module 
over $R$ of rank $2$, where $R:=\Zp[\pi_{A_0}]=\Z_p[\sqrt{p}]$. 
The Frobenius map
$\sfF$ and Verschiebung $\sfV$ operate as the multiplication by
$\sqrt{p}$. The quasi-polarization induced by $\lambda_0$ is a perfect
alternating pairing $\<\,, \,\>:M_0 \times M_0\to \Zp$ such that
$\<ax,y\>=\<x, ay\>$ for all $a\in R$ and $x,y\in M_0$. We claim 
that there exists an $R$-basis $e_1,e_2$ for $M_0$ such that 
\begin{equation}\label{eq:3.4}
  \<e_1,e_2\>=\<\sqrt{p} e_1, \sqrt{p} e_2\>=0 \quad \text{and}
  \quad \<e_1,\sqrt{p} e_2\>=\<\sqrt{p} e_1,e_2\>=1. 
\end{equation}
That is, $\{e_1,\sqrt{p}e_2, \sqrt{p} e_1, e_2\}$ is a Lagrangian
$\Z_p$-basis for
the symplectic pairing $\<\, ,\, \>$ on $M_0$. 

Let $K:=\Qp[\sqrt{p}]$ be the fraction field of $R$. Let 
\[ \<\,, \,\>_K: M_0\times M_0 \to R^\vee, \quad R^\vee:=\{x\in
K|\tr_{K/\Qp}(xR)\subset \Z_p\}=(2\sqrt{p})^{-1} R. \]
be the unique $R$-bilinear alternating form such that
$\<x,y\>=\tr_{K/\Q}\<x,y\>_K$, where $\tr_{K/\Qp}$ denotes the reduced
trace from $K$ to $\Qp$. Put
$\psi_K(x,y):=(2\sqrt{p})\<x,y\>_K$. Then $\psi_K:M_0\times M_0\to R$
is a perfect alternating $R$-bilinear pairing. 
Choose an $R$-basis $e_1,e_2$ for
$M_0$ such that $\psi_K(e_1,e_2)=1$. Using the formula
$\<x,y\>=\tr_{K/\Qp} (2\sqrt{p})^{-1} \psi_K(x,y)$, we check
(\ref{eq:3.4}) as follows:
\[ \<\sqrt{p} e_1,e_2\>=\<e_1,\sqrt{p} e_2\>=\tr (2 \sqrt{p})^{-1}
\sqrt{p}=\tr 2^{-1}=1, \]
\[ \<e_1,e_2\>=\tr (2 \sqrt{p})^{-1}=0, \quad \<\sqrt{p}e_1,\sqrt{p}
e_2\>=\<e_1, p e_2\>=0. \] \

\npr {\bf Step 3.} Let $\alpha: A\to A_0$ be any isogeny of degree $p$
defined over $\Fp$, and let $\lambda:=\alpha^*\lambda_0$. Then $A$ is a
superspecial abelian surface over $\Fp$ with $\pi_A^{2}=p$, and
$\lambda$ has the property $\ker \lambda=A[F]$. 

The \dieu module $M$ of $A$ fits into $\sfV M_0=\sqrt{p} M_0\subset
M\subset M_0$ with $\dim M/M_0=1$. Clearly, $M$ is an $R$-submodule so
that $\sfF^2$ acts as $p$ on $M$, and $M$ is superspecial. Put $\ol
M_0:=M_0/\sfV M_0={\rm Span} \{\bar e_1, \bar e_2 \}_{\Fp}$. Define a
pairing 
\[ (\,,\,):M_0\times M_0\to \Zp, \quad (x,y):=\<x,\sfF
y\>. \] Clearly, $(M_0,\sfV M_0)=(\sfV M_0,M_0)= p\Zp$. 
Modulo $p$, one obtains an $\Fp$-bilinear
pairing $(\, , \,):\ol M_0 \times \ol M_0\to \Fp$ satisfying 
$(\bar e_1, \bar e_2)=-(\bar e_2,\bar e_1)=1$ and $(\bar e_1,\bar
e_1)=(\bar e_2,\bar e_2)=0$. Now it is not hard to see that the
condition $\ker \lambda=A[F]$ is equivalent to $\sfV(M^t)=M$, which is
also equivalent to the condition $(\ol M,\ol M)=0$ (see the proof of
\cite[Lemma 4.3]{yu:ss_siegel}). As $\ol M$ is
generated by a vector $v=a \bar e_1+b \bar e_2$, where  $a,b\in \Fp$,
the condition $(\ol M,\ol M)=0$ follows immediately from  
$(v,v)=0$. Thus,  one proves $\ker \lambda=A[F]$. 

This completes the proof of Theorem~\ref{1.2}


\section*{Acknowledgements}
The author is grateful to  Tomoyoshi Ibukiyama for kind explanations of 
his recent works \cite{ibukiyama:quinary} and \cite{ibukiyama:type}.
The paper is completed during the author's visit at M\"unster
University; he acknowledges Urs Hartl and the institution for warm
hospitality and excellent working conditions.  
The author is partially supported by the grant 104-2115-M-001-001MY3.

\end{document}